\documentclass[12pt]{article}
\usepackage{amsfonts}
\usepackage{amsfonts}  
\usepackage{amssymb}
\usepackage{amsthm}
\usepackage{amsmath}
\usepackage{color}

\title{Existence and regularity results for Fully Non Linear Operators on the model of  the  pseudo  Pucci's operators  } 

\author{I. Birindelli, F. Demengel}

\date{}

\catcode`@=11 \@addtoreset{equation}{section} \catcode`@=12

\newtheorem{theo}{Theorem}[section]
\newtheorem{prop}[theo]{Proposition}
\newtheorem{rema}[theo]{Remark}

\newtheorem{cor}[theo]{Corollary}
\newtheorem{lemme}[theo]{Lemma}

\def\R{\mathbb  R}
\def\grad{\nabla}

\DeclareMathOperator*{\Lip}{Lip}

\begin{document}
\maketitle

\section{Introduction} 
This paper is devoted  to the existence and regularity of viscosity solutions for a  class of degenerate operators, on the model of the pseudo $p$-Laplacian. 

Recall that  the pseudo-$p$-Laplacian, for $p>1$ is defined by: 
$$
\tilde{\Delta}_p u:=  \sum_1^N \partial_i ( |\partial_i u|^{p-2} \partial_i u).
$$
When $p>2$, it  is degenerate elliptic at any point where even only  one derivative 
$\partial_i u$ is zero. 

Using classical methods in the calculus of variations,  equation  
\begin{equation} \label{eq1}
  \tilde{\Delta}_p u=(p-1)f
\end{equation}
has solutions in $W^{1,p}$ ,  when for example $f\in L^{p^\prime}$. The 
regularity results  are obtained through specific variational technics,  see \cite{T}, \cite{DB}. When $p<2$, Lipschitz regularity is a consequence of  \cite{FFM}. 

When $p>2$ things are more delicate.   
Note that in \cite{BC}, for some fixed  non negative numbers $\delta_i$,
the following widely degenerate equation was considered
   \begin{equation}
   \label{eq2}\sum_i \partial_i((|\partial_i u|-\delta_i)_+^{p-1} {\partial_i u \over |\partial_i u|} ) = (p-1) f.
   \end{equation} 
The authors proved that the  solutions of (\ref{eq2}) are in $W^{1,q}_{loc}$ when $f\in L^\infty_{loc}$. As a consequence, 
by the Sobolev Morrey's imbedding, the solutions are H\"older's continuous for any exponent $\gamma <1$. 
   
    The Lipschitz  interior regularity   for (\ref{eq1}) has been very recently proved by the  
second author in \cite{D}. The regularity obtained concerns Lipschitz continuity for 
viscosity solutions. Since weak solutions are viscosity solutions, (see also 
\cite{BK}), she  obtains  Lipschitz continuity for  weak solutions  when the forcing term 
is in  $L^\infty_{loc}$.
      
At  the same time,  in \cite{BBJ},  the local 
Lipschitz  regularity of the solutions of (\ref{eq2}) has been proved
when either $N = 2$,  $p\geq 2$  and $f\in W^{1,p^\prime}_{loc}$ or  $N\geq 3$, $p\geq 4$, and $f\in W^{1,\infty}_{loc}$. Remark that  (\ref{eq2}) can also be written 
formally as 
       $$\sum_i (|\partial_i u|-\delta_i)_+^{p-2} \partial_{ii} u =  f.$$
Hence         viscosity solutions have an obvious definition, and 
  with   the methods employed in \cite{D}, one can prove,  in 
particular,  that the solutions are H\"older's continuous for any exponent $\gamma <1$. Unfortunately  the Lipschitz continuity for   viscosity solution of (\ref{eq2}) cannot be obtained in the same way. 
        
Let us state the precise assumptions that hold in this paper  and present our main result.         
Fix $\alpha >0$, and                   
for any $q\in\R^N$ let  $\Theta_\alpha(q) $ be  the diagonal matrix with  entries 
$|q_i|^{\frac{\alpha}{2}}$ on the diagonal,  and let $X$ be a symmetric matrix.

\medskip

Let $F$ be defined on $\R^N \times \R^N\times S$,  continuous in all its arguments, which satisfies   $F(x, 0, M) = F(x, p, 0) = 0$ and 

(H1)  For any $M\in S$ and $N \in S$, $N\geq 0$, for any $x\in \overline{\Omega}$
  \begin{equation}\label{degenel}
    \lambda tr( \Theta _\alpha (q)N \Theta_\alpha (q))   \leq F(x, q, M+ N)-F(x, q,M) \leq    \Lambda tr(\Theta _\alpha (q)N \Theta_\alpha (q)) 
    \end{equation}
    
(H2) There exist $\gamma_F\in ]0,1]$ and $c_{\gamma_F}>0$ such that for any $(q, X)\in \R^N \times S$ 
 \begin{equation}\label{gammaF}
  |F(x, q, X)-F(y, q, X) | \leq  c_{\gamma_F} |x-y|^{\gamma_F} |q|^\alpha |X|
  \end{equation} 
(H3) There exists $\omega_F$ a continuous function on $\R^+$ such that $\omega_F(0)=0$, and 
  as soon as  $(X,Y)$ satisfy for some $m>0$
  \begin{equation}\label{eqXY}-m \left(\begin{array}{cc}
 {\rm I}&0\\
 0& {\rm I}\end{array}\right) \leq \left(\begin{array}{cc}
 X&0\\
 0& Y\end{array} \right)\leq m \left( \begin{array}{cc}{\rm I}&-{\rm I}\\
 -{\rm I}& {\rm I}\end{array} \right)
 \end{equation}
 then 

 \begin{equation}\label{depenx}F(x,m(x-y), X) -F(x, m (x-y),  Y) \leq \omega_F (m |x-y|^{\alpha+2\over {\alpha+1}} ) + o(m |x-y|^{\alpha+2\over {\alpha+1}} )\end{equation}
when $m $ goes to infinity.
  
 (H4) There exists $c_{F}$ such that
    for any $p,q\in \R^N$, for all $x\in \R^N$, $X \in S$
    \begin{equation} \label{diffp}
    |F(x,  p, X)-F(x, q, X)| \leq c_F ||p|^\alpha -|q|^\alpha | | X|
    \end{equation}

Note that the pseudo - Pucci's operators, for $0< \lambda < \Lambda$ 
 \begin{eqnarray*}
  {{\cal M}_\alpha ^+} (q, X) &=&\Lambda tr( (\Theta_\alpha(q)  X \Theta_\alpha (q))^+) - \lambda tr(( \Theta _\alpha (q)X \Theta_\alpha (q))^-) \\
  &=& \sup_{\lambda I\leq A\leq \Lambda I} tr( A  \Theta_\alpha(q)  X \Theta_\alpha(q) ).
  \end{eqnarray*}
and  
   $${{\cal M}_\alpha ^-} (q, X) = -{{\cal M}_\alpha ^+} (q, -X)$$
satisfy all the assumptions above.
     
   We will also consider  equations with lower order terms. Precisely,
let $h$ defined on $\R^N\times \R^N$ ,  continuous with respect to its arguments, which satisfies on any bounded domain $\Omega$ 
  \begin{equation}\label{H}
|h(x, q)|\leq c_{h, \Omega}  (|q|^{1+\alpha}  +1)
  \end{equation}
  
Our main result is the following.
\begin{theo}\label{lip} Let $\Omega$ be a bounded domain and $f$ be continuous and bounded
in $\Omega$. Under the conditions (\ref{depenx}),  (\ref{degenel}), (\ref{gammaF}), (\ref{diffp}) and (\ref{H}), let $u$ be  a  solution of
\begin{equation} \label{equation}
F(x,\grad u, D^2u)+h(x,\grad u)=f \quad\mbox{in }\ \Omega.
\end{equation}
 Then, for  any $\Omega^\prime \subset \subset \Omega$ ,
there exists $C_{\Omega^\prime}$,  such that   
for   any $(x,y)\in \Omega^\prime$ 
$$|u(x) -u(y)|\leq C_{\Omega^\prime}  |x-y|.$$
\end{theo}
This will be a consequence of the more general result Theorem \ref{lip1} in section three.

We shall construct in Section 4 a  super-solution  of (\ref{equation})
which is  zero on the boundary. Ishii's Perron method, since the comparison principle holds, leads to the following existence's result : 
  \begin{theo} \label{corexi}
  Suppose that $\Omega$ is a bounded ${\cal C}^2$ domain and let $F$ and $h$ satisfy   (\ref{degenel}), (\ref{depenx}), (\ref{diffp}),  (H3), and  (\ref{H}). Then  for any $f\in {\cal C} (\overline{\Omega})$ there exists $u$ a viscosity solution of 
  $$\left\{ \begin{array}{lc}
   F(x, \nabla u, D^2u) + h(x, \nabla u) = f(x)&{\rm in} \ \Omega
\\
u=0 & {\rm on} \ \partial \Omega
\end{array} \right.$$
 Furthermore $u$ is Lipschitz continuous in $\Omega$ . 
  \end{theo}   
Finally in the last section we prove that the strong maximum principle holds.

   Let us end this introduction by saying a few words about the principal eigenvalues and eigenfunctions, on the model of \cite{BNV}. Indeed the regularity and existence results obtained above allow to prove the existence of a {\em principal eigenvalue} as long as the operator  ,  $F$ is in addition homogeneous,  precisely :
   
 For any $(x,p,X)$ and any $s\in\R$ and $t\geq 0$:
$$F(x, s p, t X) = |s|^\alpha t F(x,p, X). 
$$
 We  also suppose that $h$  is continuous with values in $\R^N$, and that $h(x, p) = h(x) \cdot p |p|^\alpha$. 
 Then we can  define  the two values   
 
$$\mu^+ = \{ \mu\in \R, \\ \exists \varphi >0\ {\rm in}\ \Omega,\ F(x, \grad\phi,D^2\phi) + h(x) \cdot \nabla \phi |\nabla \phi|^\alpha +\mu
\phi^{\alpha+1}\leq 0\ \}.$$
and  $$\bar\mu^-:=\sup\{ \mu\in \R,\ \ \exists \  \psi< 0\  
\ {\rm in}\ \Omega,\ F(x, \grad\psi,D^2\psi) + h(x) \cdot \nabla \psi |\nabla \psi|^\alpha +\mu
|\psi|^{\alpha}\psi \geq 0\ \}.$$
As in previous works e.g. \cite{BNV}, \cite{BEQ} and \cite{BD1}  it is easy to prove that below $\mu^+$ the classical property of maximum principle holds, i.e.
if $\tau<\bar \mu^+$ and $u$ is a solution of 
 $$ F(x, \nabla u, D^2 u) + h(x) |\nabla u| ^\alpha \nabla u + \tau |u|^\alpha u \geq 0$$
such $u\leq 0$ on $\partial \Omega$ then $u\leq 0$ in $\Omega$.
Similarly, for any $\tau<\mu^-$, the minimum principle holds.

Furthermore, one can prove the existence
of $\psi^+>0$ and $\psi^-<0$ solution, respectively of 
  $$F (x, \nabla \psi^+, D^2 \psi^+) + \mu^+ (\psi^+)^{\alpha+1} = 0\ \mbox{in }\ \Omega,\ \psi^+=0\  
    \mbox{on }\ \partial\Omega, $$
 $$F (x, \nabla \psi^-, D^2 \psi^-) + \mu^- |\psi^-|^{\alpha}\psi^- = 0\ \mbox{in }\ \Omega,\ \psi^-=0\  
    \mbox{on }\ \partial\Omega.$$
Thus the values  $\mu^+$ and $\mu^-$ are called "principal eigenvalues".
 
We will not give proofs for these last results which can be obtained arguing as 
in \cite{BD1}, and using the comparison principle in Theorem \ref{comp} and the 
Lipschitz a priori bounds.

Many questions concerning these very degenerate operators are still  open . To name a few
let us mention:

Does Alexandroff Bakelman Pucci 's inequality hold true,  similarly to the cases treated in \cite{I}? 
     
Is some Harnack inequality true (still as in \cite{I})? Even for the pseudo-$p$-Laplacian this is not known. 
Finally the  further step  in regularity is naturally the  ${\cal C}^1$ regularity. 
Even in the case $f=0$ and $N=2$ it does not seem easy to obtain.

 \section{Examples }  
  
{\bf Example 1} : 
Let  
$$ F(x, p, X) := tr( L(x)\Theta_\alpha(p) X\Theta_\alpha(p)   L(x)). $$
if $L(x)$ is a Lipschitz and bounded  matrix  such that $\sqrt{\lambda}{\rm I} \leq L\leq \sqrt{\Lambda}{\rm I} $ then conditions (\ref{degenel}) and (\ref{gammaF}) are obviously satisfied.
In order to check the  condition (H3), one uses the  right inequality in
(\ref{eqXY}) multiplied by
 $\left( \begin{array}{cc}
 L(x)  \Theta_\alpha(p)\\
  L(y) \Theta_\alpha(p) \end{array} \right)$ on the right and by  its transpose on the left.  
 Hence

  \begin{eqnarray*}
  F(x, m(x-y), X) &-&F(x,   m (x-y), -Y)\\
  & \leq&  m^t(L(x) -L(y)) \Theta_\alpha^2(m(x-y))  (L(x)-L(y)) \\
  &\leq&  m^{\alpha+1} |x-y|^{2+\alpha} \\
    \end{eqnarray*}
Let us check now condition (\ref{diffp})
     \begin{eqnarray*}
   F(x,p , X)  &-&  F(x, q, X)
      =  
      tr( L(x) (\Theta_\alpha(p)-\Theta_\alpha (q) ) X  (\Theta_\alpha (p)+ \Theta_\alpha (q))L(x)) \\
      &\leq & \Lambda |X| |\Theta_\alpha (p)^2-\Theta_\alpha(q)^2| 
      \end{eqnarray*}
which yields the result.   
     
\medskip     
{\bf Example 2 }: We define 
$$ F(x, p, X) := a (x) {{\cal M}_{ \alpha}^\pm} (p, X). $$
If $a$ is a Lipschitz function such that
$a(x) \geq a_o>0$ then conditions (\ref{degenel}) and (\ref{gammaF}) are satisfied. 
Let us check the 
condition (H3).
    
Recall the following standard extremality property of the Pucci's operators 
     
      $${\cal M}^+ (X) \leq {\cal M}^+ (-Y) + {\cal M}^+(X+Y)\ \mbox{
       and }\ {\cal M}^- (X) \leq {\cal M}^- (-Y) + {\cal M}^+(X+Y).$$
Using  the identity  

$$ a (x) {{\cal M}_{ \alpha}^\pm} (X) = {\cal M}^\pm (a(x) \Theta_\alpha(p)  X \Theta_\alpha(p) ) = {\cal M}^\pm ( \sqrt{a(x) } \Theta_\alpha(p)  X \Theta_\alpha(p)   \sqrt{a(x) } )$$
we  have 
 
  \begin{eqnarray*}
 {\cal M}^\pm ( \sqrt{a(x) }  \Theta_\alpha(p)  X \Theta_\alpha(p)   \sqrt{a(x) })&\leq&  {\cal M}^\pm ( \sqrt{a(y) }  \Theta_\alpha(p)  (-Y) \Theta_\alpha(p)   \sqrt{a(y) } )\\
 &+ &{\cal M}^+ \left[ \sqrt{a(x) }  \Theta_\alpha(p)  X \Theta_\alpha(p)   \sqrt{a(x) }\right.\\
 &+& \left. \sqrt{a(y) }  \Theta_\alpha(p) Y \Theta_\alpha(p)   \sqrt{a(y) })\right].
 \end{eqnarray*}
Multiplying  (\ref{eqXY}), by  the matrix 

$$\left(\begin{array} {cc}
  \sqrt{a(x) } \Theta_\alpha(p) & 0\\
  0& \sqrt{a(y) }  \Theta_\alpha(p) 
  \end{array} \right)$$
on the left and on the right, one obtains that  for $p = m(x-y)$, 
\begin{eqnarray*}
    \sqrt{a(x) }  \Theta_\alpha(p)  X \Theta_\alpha(p)  \sqrt{a(x) } &+ & \sqrt{a(y) }  \Theta_\alpha(p) Y \Theta_\alpha(p)   \sqrt{a(y) } \\
    &\leq& 
   m (\sqrt{a(x)}-\sqrt{a(y)})^2 \Theta_\alpha(p)  ^2\\
   & \leq& m^{\alpha+1} |x-y|^{\alpha } {(a(x)-a(y))^2\over( \sqrt{a(x)}+ \sqrt{a(y)})^2}I \\
   & \leq& (\Lip a)^2  {m^{\alpha+1} |x-y|^{\alpha+2}\over 4 a_o}I .
   \end{eqnarray*}
In particular 
    
   \begin{eqnarray*}
   {\cal M}^+ ( \sqrt{a(x) }  \Theta_\alpha(p)  X \Theta_\alpha(p)   \sqrt{a(x) } &&+  \sqrt{a(y) }  \Theta_\alpha(p)  Y \Theta_\alpha(p)   \sqrt{a(y) } )\\
   &\leq&\Lambda  (\Lip a )^2N   {m^{\alpha+1} |x-y|^{\alpha+2}\over 4 a_o}.
   \end{eqnarray*}
Let us check finally (\ref{diffp}) ,  for that, it is clear that one can suppose $a(x)=1$, 
we write

\begin{eqnarray*}
|{\cal M}_\alpha ^\pm ( p, X)-{\cal M}^\pm _\alpha (q, X) |&\leq& {\cal M}^+ ( \Theta_\alpha(p)  X \Theta_\alpha(p)  -\Theta_\alpha (q) X \Theta_\alpha (q)) \\
&=&{1\over 2} {\cal M}^+ \left[  (\Theta _\alpha(p) )- \Theta_\alpha (q) )X (  \Theta_\alpha (p)+  \Theta_\alpha (q))\right. \\
&+& \left. (  \Theta_\alpha (p)+  \Theta_\alpha (q)) X (  \Theta_\alpha (p)-  \Theta_\alpha (q))\right]\\
&\leq & \Lambda  |\Theta_\alpha^2 (p)- \Theta_\alpha^2 (q)| |X| 
\end{eqnarray*}
    \section{Proof of Lipschitz regularity.}
In this section we prove our main result:
\begin{theo}
 \label{lip1}
Let $f$ and $g$ be continuous and bounded in $\Omega$, while $F$, $\Omega$  and  $h$ satisfy  the hypothesis in Theorem \ref{lip}.  Suppose that $u$ is a   bounded USC sub-solution of 
$$ F(x, \nabla u, D^2 u) + h(x, \nabla u)\geq f\ \mbox{in}\ \Omega$$
and  $v$ is a  bounded LSC super-solution of 
 $$F(x, \nabla v, D^2 v)+ h(x, \nabla v) \leq g\ \mbox{in}\ \Omega.$$
Then for  any $\Omega^\prime \subset \subset \Omega$  
there exists $C_{\Omega^\prime}$,  such that   
for   any $(x,y)\in (\Omega^\prime)^2$ 
$$u(x) \leq v(y) + \sup_{\Omega}  (u-v) + C_{\Omega^\prime}  |x-y|.$$
\end{theo}
We start by recalling some general facts.

If $\psi:\R^N\times\R^N\rightarrow \R$, let $D_1\psi$ denotes the gradient in the first 
$N$  variables and $D_2\psi$ the gradient in the last $N$ variables.

In the proof of Theorem \ref{lip1} we shall need the following technical lemma.

\begin{lemme} \label{lem2} Suppose that $u$ and $v$ are respectively  USC and LSC functions such that, for some constant $M>1$ and   for some function $\Phi$         \[     
            u(x)-v(y) -M  |x-x_o|^2 -M  |y-x_o|^2- \Phi(x,y)
\]
has a local maximum in $(\bar x,\bar y)$ where $\phi$ is ${\cal C}^2$.

Then for any $\iota$,   there exist $X_\iota ,Y_\iota$ such that 
\begin{eqnarray*}
 (D_1 \Phi(\bar x, \bar y)+ 2M (\bar x-x_o), X_\iota) \in \bar J^{2,+} u(\bar x), \\ 
 (-D_2 \Phi(\bar x, \bar y) -2M (\bar y-x_o), -Y_\iota) \in \bar J^{2,-} v(\bar y)
\end{eqnarray*}
with 
$$ -({1\over \iota} + |A|+1 ) \left(\begin{array}{cc}
           {\rm I} &0\\
           0&{\rm I}\end{array}\right) \leq \left(\begin{array}{cc}
           X_\iota-2M{\rm I}&0\\
           0& Y_\iota-2M {\rm I}
           \end{array}\right)\leq (A+\iota A^2)    +  \left(\begin{array}{cc}
           {\rm I} &0\\
           0&{\rm I}\end{array}\right)         
$$
and $A = D^2 \Phi(\bar x, \bar y)$.
\end{lemme}

This is a direct consequence of a famous Lemma by Ishii  \cite{I1}. For the convenience of the reader the proof of Lemma \ref{lem2} is given in the appendix.  In the sequel, for 
some $M$, we will use  Lemma \ref{lem2} with  
$\Phi(x,y) :=Mg(x-y)$, where $g$ is some functions which is ${\cal C}^2$ except at $0$,  to be   defined 
later.  
Denoting by $H_1(x) := D^2 g(x)$,  then   
             
       $$ D^2\Phi =M \left(\begin{array}{cc} H_1(\bar x-\bar y)& -H_1(\bar x-\bar y)\\
             -H_1(\bar x-\bar y)& H_1(\bar x-\bar y)\end{array} \right) $$ Choosing   $\iota = {1\over 4M |H_1(x)|}$, and  defining 
$\tilde H(x) := H_1 (x)+ {2\over 4 |H_1(x)|}  H_1^2(x)$, one has 
$$\ D^2\Phi + \iota (D^2\Phi)^2 =    M\left(\begin{array}{cc} \tilde H(\bar x-\bar y)& -  \tilde H(\bar x-\bar y)\\
 -\tilde H(\bar x-\bar y)&  \tilde H(\bar x-\bar y)\end{array} \right). $$ 
Remark that $|A| = 2M |H_1(\bar x-\bar y)|$. 
We give some precisions on the choice of $g$: 
 We will assume that $g$ is radial, say there exists $\omega$ some   continuous  function on $\R^+$,  such that $g(x) = \omega (|x|)$ and   $\omega$ is supposed to satisfy : 
 
  \begin{equation}
  \label{omega}
  \omega(0)=0, \ \omega \ {\rm is} \ {\cal C}^2 \mbox{on} \ \R^{+\star}, \ \omega (s) >0, \   \omega^\prime (s)>0\ {\rm  and} \  \omega^{\prime \prime} (s) <0 \ {\rm on} \ ]0,1[.
  \end{equation} For $x\neq 0$, it is well known that 
 $Dg (x)= \omega^\prime (|x|) {x\over |x|}$
  and 
  $$D^2 g (x) = \left(\omega^{\prime \prime } ( |x|) -{\omega^\prime ( |x|)\over  |x|}\right) {x\otimes x\over  |x|^2} +  {\omega^\prime ( |x|)\over  |x|} {\rm I}. $$  For $\iota \leq    {1\over 4 | D^2g(x)| } $, there exist constants $\gamma_H\in ]{1\over 2} , {3\over 2}], \  \beta_H \geq {1\over 2}$ such that 
   \begin{equation} \label{alphabetaH}
   D^2g+2 \iota  (D^2g)^2 (x)=  \left(\beta_H\omega^{\prime \prime } ( |x|) -\gamma_H{\omega^\prime ( |x|)\over  |x|}\right) {x\otimes x\over  |x|^2} + \gamma_H {\omega^\prime ( |x|)\over  |x|} {\rm I}. 
   \end{equation} For $|x|<1$ and $\epsilon>0$, we shall use the following set:
$$
I(x, \epsilon) := \{ i\in [1,N] , |x_i | \geq  |x|^{1+\epsilon} \}. 
$$ 
 We  also define the diagonal matrix $\Theta(x)$ with  entries 
$\Theta_{ii}(x) =\left\vert { \omega^\prime (|x|) x_i\over |x|}\right\vert^{\alpha \over 2}$. 
   
A consequence of (\ref{alphabetaH}) is the following Proposition proved in \cite{D}.
\begin{prop}[\cite{D}]\label{prop4}      
  \noindent  1) If $\alpha \leq 2$,  for all $x\neq 0$,   $|x|<1$,   $\Theta (x) \tilde H (x)\Theta (x)$
 has at least one eigenvalue smaller than 
 \begin{equation}\label{pleq4} N^{-\alpha\over 2}\  \beta_H\omega^{\prime\prime} (|x|)(\omega^\prime (|x|) )^{\alpha }.
 \end{equation}
 2)  If $\alpha>2$, for all $x\neq 0$, $|x|<1$,   for any   $\epsilon >0$    such that $I(x, \epsilon) \neq \emptyset$,  and  such that 
 \begin{equation}\label{eqNepsilon}
\beta_H\omega^{\prime\prime} (|x|) (1-N |x|^{2\epsilon} )+\gamma_H N  |x|^{2\epsilon}{ \omega^{\prime }(|x|)\over |x|} \leq  {\omega^{\prime\prime} (|x|)\over 4}  <0 , 
  \end{equation} 
then  $\Theta (x) \tilde H (x)\Theta (x)$ possesses at least one eigenvalue smaller than 
\begin{eqnarray}\label{p>4}
 {1-N |x|^{2\epsilon}\over \# I(x, \epsilon)} (\omega^\prime (|x|))^{\alpha }\frac{\omega^{\prime\prime} (|x|)}{4}|x|^{(\alpha-2)\epsilon}  .
\end{eqnarray}\end{prop}

[Proof of Theorem \ref{lip1}]
Borrowing ideas from \cite{IS}, \cite{BCI}, \cite{BD2}, for some $x_o\in B_r$ we define
 the function 
     $$\psi(x,y) = u(x)-v(y)-\sup (u-v) - M \omega (|x-y|) -M |x-x_o|^2 -M |y-x_o|^2;$$
 $M$ is a large constant  and  $\omega$  is a function satisfying (\ref{omega}), both to be defined more precisely later . 

If there exists $M$  independent of $x_o\in B_r$ such that $\psi(x,y)\leq 0$ in $B_1^2$, by  taking  $ x= x_o$ and using $|x_o-y| \leq 2$ one gets 
      $$ u(x_o)-v(y)\leq \sup (u-v)+ 3M \omega (|x_o-y|).$$
       So making $x_o$ vary we obtain that for any $(x,y)\in B_r^2$
       $$ u(x) -v(y) \leq \sup (u-v)+ M \omega (|x-y|).$$
      
  This proves the theorem when $\omega$ behaves like $s$ near zero.  This can be done once the case where $\omega (s)  = s^\gamma$ is treated  for  $\gamma \in ]0,1[$, i.e the H\"older's analogous result.

\medskip
In order to prove that $\psi(x,y)\leq 0$ in $B_r^2$, suppose by contradiction that the supremum of 
$\psi$, achieved  on $(\bar x, \bar y)\in \overline{B_1}^2$, is positive.
If 
we have chosen $M$ such that 
\begin{equation} \label{Mr}
M(1-r)^2 > 4 (|u|_\infty + |v|_\infty), \ {\rm and } \ M > {2 |u|_\infty + 2 |v|_\infty \over \omega (\delta)},
\end{equation}
we would get that 
$|\bar x-x_o| , |\bar y-x_o| < {1-r\over 2}$.  
Hence,  by (\ref{Mr}), $\bar x$ and $\bar y$ are in $B_{1+r\over 2}$ i.e. in $B_1$. 
Furthermore, always using (\ref{Mr}), the positivity of the supremum of $\psi$  
leads to $|\bar  x -\bar y| < \delta$. 

As it is shown later the contradiction will be found by choosing $\delta$ small enough and $M$ large 
enough depending on $(r,\alpha, \lambda, \Lambda, N)$.

We proceed using Lemma \ref{lem2} and so, for all $\iota>0$ there exist 
$ X_\iota$ and $Y_\iota$ such that
$$(q+2M(x-x_o), X_\iota ) \in \overline{J}^{2,+} u(\bar x)\ \mbox{ and}  \
     (q-2M(y-x_o), -Y_\iota )\in  \overline{J}^{2,-}v(\bar y)$$ with 
$q = M \omega^\prime (|\bar x-\bar y|) \frac{\bar x-\bar y}{|\bar x-\bar y|} $. 
Furthermore, still using the above notations i.e. $g(x) = \omega (|x|)$,  
and  choosing $\iota = {1\over 4M|D^2 g(x)|} $,  for  $\bar H=(D^2 g(x)+ {1\over 2 |D^2 g(x)|} D^2 g(x)^2 ) $, we have that
     
\begin{eqnarray}\label{ine}
      -   ( {1\over \iota} +M|\bar H|)\left(\begin{array} {cc}
     I& 0\\
     0& I\end{array} \right) 
     &\leq& \left(\begin{array}{cc}
     X_\iota-(2M+1) {\rm I}   & 0\\
     0& Y_\iota-(2M+1) {\rm I} \end{array} \right) \nonumber \\
     &\leq& M\left( \begin{array}{cc} \bar H  & -  \bar H\\
     -\bar H& \bar H
     \end{array} \right).
     \end{eqnarray} 

 From now on we will drop the 
$\iota $ for $X$ and $Y$. Recall that $\Theta(q)$ is the diagonal matrix such that $(\Theta )_{ii}(q) = \left({|q_i|} \right)^{\alpha \over 2}$.

 In order to end the proof we will prove the following claims. 
        
{\bf Claims.} {\em There exists $\hat \tau >0$, such that, if $\delta$ is small enough and $|x-y| < \delta$  the matrix $ \Theta ( X+Y) \Theta$ has one eigenvalue $\mu_1$ such that
\begin{equation}\label{vpneg}
        \mu_1(\Theta ( X+Y) \Theta) \leq   -c M^{\alpha+1}|\bar x-\bar y|^{-\hat \tau}
\end{equation}
There exist $\tau_i<\hat \tau$ and $c_i$ for $i=1,\dots, 4$ such that  the four following assertions hold : 
  \begin{equation} \label{autresvp}
\mbox{ for  all} \   j\geq 2 \ \mu_j (\Theta (X+Y) \Theta) \leq 
c_1 M^{\alpha+1}|\bar x-\bar y|^{-\tau_1},
\end{equation} 
     
\begin{equation}\label{eqqx}
          |F(\bar x, q^x, X)- F(\bar x, q, X)|+    |F(\bar y, q^y, -Y)- F(\bar y, q, -Y)|\leq c_2 M^{\alpha+1}|\bar x-\bar y|^{- \tau_2}
\end{equation}

\begin{equation} \label{eqF}
          |F(\bar x, q, X)-F(\bar y, q , X)| + |F(\bar x, q, -Y)- F(\bar y, q, -Y)| \leq c_3 M^{1+\alpha} |\bar x-\bar y|^{-\tau_3},
          \end{equation}
          \begin{equation}\label{eqh}
         |  h(\bar  x, q^x)| + |h(\bar y, q^y) |  \leq c_4 M^{1+\alpha} |\bar x-\bar y|^{-\tau_4}.
           \end{equation} }
           
           From all these claims,  by taking $\delta$ small enough depending only on $c_i$ and $\tau_i$, $\lambda, \Lambda, \alpha, N, r$  one gets

     $$F(\bar x, q^x, X) -F(\bar y, q^y, -Y)+  h(\bar x,q^x)-h(\bar y, q^y) \leq  -{\lambda c\over 2} M^{\alpha+1}|\bar x-\bar y|^{-\hat \tau}.
     $$     
Precisely one needs to take $\delta$ such that 
            $ c_2 \delta^{-\tau_2+\hat \tau} + c_3 \delta^{\tau-\tau_3}+ c_4 \delta^{\tau-\tau_4} + \Lambda c_1 \delta^{\tau-\tau_1} < {\lambda  c\over 2}$.     
Finally, one can conclude as  follows 
       
\begin{eqnarray*}
       f(\bar x) &\leq& F(\bar x, q^x, X)+ h(\bar x, q^x) \\
       &\leq & F(\bar y, q^y, -Y)+ h(\bar y, q^y)-{\lambda c\over 2} M^{\alpha + 1}|\bar x-\bar y|^{-\hat \tau}\\
       &\leq &-{\lambda c\over 2}  M^{\alpha + 1}|\bar x-\bar y|^{-\hat \tau} + g(\bar y).
\end{eqnarray*}
This contradicts the fact that $f$ and $g$ are bounded, 
as soon as $\delta$ is small or $M$ is large enough.  
And then in order to get the desired result it is sufficient to prove in all cases  (\ref{vpneg}), (\ref{autresvp}), (\ref{eqqx}), (\ref{eqF}),  (\ref{eqh}).

So to prove the claims, we will use  inequality (\ref{ine})  which has three  important consequences for $\Theta( X+Y-2(2M+1)  {\rm I}) \Theta $:

\begin{enumerate}
\item As is well known the second inequality in (\ref{ine}) gives \\
$(X+Y-2(2M+1){\rm I}) \leq 0$, then  also  $  \Theta (X+Y-2(2M+1) {\rm I}) \Theta \leq 0$.
In particular\begin{equation}\label{autrevp}
\mbox{ all the eigenvalues of } \  \Theta (X+Y) \Theta \ \mbox{ are less than} 
 \ 6 M  |\Theta |^2.
 \end{equation}
   
\item By Proposition \ref{prop4}, $\Theta (\bar H) \Theta$ has a large negative eigenvalue, 
let $e$ be the corresponding eigenvector. 
Multiplying by     $\left(\begin{array}{c} e\\ -e\end{array} \right)$ on the right and   by its transpose on the left 
of (\ref{ine}), one gets, using (\ref{pleq4}) that, for some positive 
constant $c$, when  $\alpha \leq 2$,
    \begin{equation}\label{ppvp}
     \mu_1 ( \Theta (X+Y-2(2M+1) {\rm I} ) \Theta ) \leq c M^{1+\alpha}\omega^{\prime \prime }(|\bar x-\bar y|)  (\omega ^\prime (|\bar x-\bar y|) )^\alpha;
     \end{equation} 
this in particular implies that $$
\mu_1 ( \Theta (X+Y) \Theta )  \leq c M^{1+\alpha}\omega^{\prime \prime }(|\bar x-\bar y|)  (\omega ^\prime (|\bar x-\bar y|) )^\alpha +6M|\Theta|^2.
$$
When $\alpha \geq 2$, if (\ref{eqNepsilon}) holds, using (\ref{p>4}),
      \begin{equation}\label{ppvp1} \mu_1 (  \Theta (X+Y ) \Theta ) \leq c M^{1+\alpha}  \omega^{\prime \prime }(|\bar x-\bar y|)  (\omega ^\prime (|\bar x-\bar y|))^\alpha|\bar x-\bar y|^{2\epsilon}+ 6M |\theta |^2.
      \end{equation}

\item Finally, using (\ref{ine}), we obtain an upper bound for $|X|+ |Y|$
i.e.
\begin{equation}\label{XY}
 |X| + |Y| \leq C M( |D^2g|+1).
 \end{equation}
 remarking that $|\bar H| \leq {3\over 2} |D^2 g(x)|$. 
\end{enumerate}

We will need to detail the cases  $\omega(s)\simeq s$ or $\omega(s)= s^{\gamma}$ both when   $\alpha \leq 2$ or $\alpha \geq 2$.

\medskip

\noindent {\bf  Proofs of  the claims when  $\omega(r)=r^\gamma$ and $\alpha \leq 2$.}

In this case,  $\omega (s) = s^\gamma$, $\omega^\prime (s) = \gamma s^{\gamma-1}$ and $\omega^{\prime \prime} (s) =- \gamma (1-\gamma)s^{\gamma-2}$, $q = M\gamma |\bar x-\bar y|^{\gamma-1} {\bar x-\bar y\over |\bar x-\bar y|}$, $q^x = q+ 2M (\bar x-x_o)$, $q^y = q-2M (\bar y-x_o)$. 
By    (\ref{ppvp}), since $\gamma\in (0,1)$,
$ \Theta (X+Y-2(2M+1){\rm I} ) \Theta $ has one eigenvalue less than 
$$
  -{\gamma(1-\gamma)\over 4}  M^{\alpha+1} |\bar x-\bar y|^{\gamma-2+ (\gamma-1) \alpha}.
$$

While 
$6 |\Theta |^2 \leq 6M^\alpha \gamma^\alpha  |\bar x-\bar y|^{(\gamma-1) \alpha}$. Consequently,    
as soon as $\delta$ is small enough, 
$ \Theta (X+Y ) \Theta $  has at least one eigenvalue 
less than 
$ -{\gamma(1-\gamma)\over 4}  M^{\alpha+1} |\bar x-\bar y|^{\gamma-2+ (\gamma-1) \alpha}+ M|\Theta |^2 \leq  -{\gamma(1-\gamma)\over 4}  M^{\alpha+1} |\bar x-\bar y|^{\gamma-2+ (\gamma-1) \alpha}+6M^{1+\alpha} \gamma^\alpha  |\bar x-\bar y|^{(\gamma-1) \alpha}\leq   -\gamma{1-\gamma\over 8} M^{\alpha+1} |\bar x-\bar y|^{\gamma-2+ (\gamma-1) \alpha}$.
This proves (\ref{vpneg}) with $\hat\tau=2-\gamma+ (1-\gamma) \alpha>\tau_1$, and $c = \gamma{1-\gamma\over 8} $. 
 
 Now using (\ref{autrevp}) and   the above estimate on $M |\Theta|^2$, 
(\ref{autresvp}) holds with $\tau_1 =( 1-\gamma) \alpha$.

Note now that $$ |D^2 g(\bar x-\bar y)| \leq  \gamma (N-\gamma)  |\bar x-\bar y|^{\gamma-2}, $$  and recall that 
$|\bar H | \leq {3\over 2} |D^2 g|$, and 
then, by (\ref{XY}),
\begin{equation}\label{XYhold}
|X|+ |Y| \leq 6\gamma (N-\gamma+3) M|\bar x-\bar y|^{\gamma-2}.
\end{equation}

Consequently (\ref{eqF}) holds with   $\tau_2 = (2-\gamma) + (1-\gamma) \alpha-\gamma_F$ and $c_2 = 12 c_{\gamma_F} \gamma^{1+\alpha} (N+3-\gamma)$ using hypothesis (\ref{gammaF}).   
       
To prove (\ref{eqqx}) we will  use the  following universal inequality :  For any  $Z$ and $T$ in $\R^N$ 
 \begin{equation}\label{ZT}
       ||Z|^\alpha -|T|^\alpha |\leq  \sup (1, \alpha) |Z-T|^{\inf (1, \alpha)} (|Z| + |T|)^{ (\alpha-1)^+}
 \end{equation} 
 in the form
       $$ ||q^x|^\alpha -|q|^\alpha| \leq 2^\alpha  \sup (1, \alpha) M^\alpha |\bar x-\bar y  |^{(\gamma-1) (\alpha-1)^+)}.$$
       Hence using (\ref{XYhold}),  (\ref{eqqx}) holds with 
        $\tau_3 = (2-\gamma) + (1-\gamma)  (\alpha-1)^+ $, and $c_3 =c_F 2^{1+\alpha}  (\gamma+1)^{ (\alpha-1)^+}$. 
         Finally (\ref{eqh}) holds with $\tau_4 = (1-\gamma) (1+ \alpha) $ and $c_4 = 2c_{h, \Omega}((\gamma+3)^{1+\alpha} + 1)$ . 
       
\medskip
\noindent{\bf  Proofs of the claims  when $\omega(r)=r^\gamma$ and $\alpha \geq 2$}. 
The function $\omega$ is the same than in the previous case. In order to use the result in Proposition \ref{prop4} we need  (\ref{eqNepsilon})  to be  satisfied.  For that aim  
we take $\epsilon >0$ such that $\epsilon <\inf ( {\gamma_F\over 2}, {1-\gamma\over 2})$. 
Let
 \begin{equation}\label{eq21}
  \delta_N := \exp ({-\log
(2N(4-\gamma))+ \log (1-\gamma) \over 2\epsilon} ),
\end{equation}
and assume $\delta < \delta_N$. 
In    particular,    since there exists $i\in [1,N]$ such that 
$$|\bar
x_i-\bar y_i|^2 \geq{ |\bar x-\bar y|^2\over N}\geq |\bar x-\bar
y|^{2+ 2\epsilon},$$ 
for $\alpha \geq 2$, using the definition of
$\delta_N$ in  (\ref{eq21}),  $I(\bar x-\bar y,\epsilon)\neq
\emptyset$. Furthermore  for $|\bar x-\bar y| < \delta \leq \delta_N$

        \begin{eqnarray*}
                  {1\over 2}   \omega^{\prime \prime } (|\bar x-\bar y|)(1-N
|\bar x-\bar y|^{2\epsilon} )  &+& {3N\over 2}  |\bar x-\bar
y|^{2\epsilon} {\omega^\prime (|\bar x-\bar y|)\over |\bar x -\bar
y|} \\
  & \leq & {1\over 2}   \omega^{\prime \prime }
(|\bar x-\bar y|) \\
&&+  {N\over 2} |\bar x-\bar y|^{2\epsilon} (\gamma
(1-\gamma) + 3\gamma) |\bar x-\bar y|^{\gamma-2}
\\
  &\leq&    {1\over 4} \gamma (\gamma-1) |\bar
x-\bar y|^{\gamma-2}  \\
&= & {\omega^{\prime \prime} (|\bar x-\bar y|)\over 4} ,
 \end{eqnarray*} 
and then (\ref{eqNepsilon}) is satisfied. 
We are in a position to apply  (\ref{ppvp1}), and    $\Theta (X+Y)   \Theta $ has at least one eigenvalue $\mu_1$  less than 
$ -(\gamma{1-\gamma\over 4})M^{\alpha+1} |\bar x-\bar y|^{\gamma-2+ (\gamma-1) \alpha+ \epsilon}+ 6 M |\Theta |^2$, hence 
  \begin{eqnarray*}
 \mu_1 
& \leq &-  ({\gamma(1-\gamma)\over 4})M^{\alpha+1} |\bar x-\bar y|^{\gamma-2+ (\gamma-1) \alpha+ \epsilon}\\
  &+& 6M^{1+\alpha} \gamma^\alpha |\bar  x-\bar y|^{(\gamma-1) \alpha}\\
  &\leq & -  ({\gamma(1-\gamma)\over 8})M^{\alpha+1} |\bar x-\bar y|^{\gamma-2+ (\gamma-1) \alpha+ \epsilon}
 \end{eqnarray*}
 for $|\bar x-\bar y| \leq \delta$ small enough, hence ({\ref{vpneg}) holds with $\hat \tau = 2-\gamma + (1-\gamma) -\epsilon$. 

 Note that (\ref{autresvp}),  (\ref{eqF})   (\ref{eqqx}) and  (\ref{eqh})  have already been proved in the previous case, since   the  sign of  $\alpha-2$  does not play a role. 
Recall then that  $\tau_1 = (-\gamma+1) \alpha$,  and    $c_1 = 6 \gamma^{1+\alpha} (N-\gamma+3)$, while 
        $\tau_2 = (2-\gamma) + (1-\gamma) \alpha-\gamma_F< \hat \tau$ by the choice of $\epsilon$, and $c_2 = 12 c_{\gamma_F} \gamma^{1+\alpha} (N+3-\gamma)$.  
    
    Finally $\tau_3 = (2-\gamma) + (\alpha-1)(\gamma-1)$ and $c_3 =c_F 2^{1+\alpha}  (\gamma+1)^{ (\alpha-1)^+}$, and  (\ref{eqh}) still holds with $\tau_4 =   (1-\gamma) (1+ \alpha) $.

\medskip
Let us observe that in the hypothesis of Theorem \ref{lip1}  we have proved that $u$ 
and $v$ satisfy, for any $\gamma\in (0,1)$,  
\begin{equation}\label{holder}
u(x) \leq v(y) + \sup_{\Omega}  (u-v) + c_{\gamma, r}|x-y|^\gamma
\end{equation}  
This will be used in the next cases. 

\noindent{\bf  Proofs of the claims  when $\omega(r)\simeq r$ and $\alpha \leq 2$.}
We choose $\tau \in (0,  \inf (\gamma_F,{1\over 2},{ \alpha\over 2}))$ and  $\gamma\in (1, {\tau\over \inf ({1\over 2},{ \alpha\over 2})})$. We define 
 $ \omega(s) = s-\omega_o s^{1+\tau}$,
 where $s < s_o=\left( {1\over (1+\tau) \omega_o}\right)^{1\over \tau}$ and 
 $\omega_o$ is chosen so that  $s_o>1$. 
We suppose that $\delta^\tau \omega_o (1+\tau) < {1\over 2}$, which ensures that 
 
\begin{equation}\label{omegaprime}
 {\rm for} \ s < \delta   \    {1\over 2} \leq \omega^\prime (s) < 1, \ \omega(s) \geq {s\over 2} .
\end{equation}
We  suppose that 

             \begin{equation}\label{Mrlip}
             M {\delta \tau\over  (1+\tau)} > 2\sup u,\ M>1 \   {\rm and}\ M ({1-r\over 2} )^2 > 2\sup u.
             \end{equation} 
which implies in particular  (\ref{Mr}).  

 Here $|D^2 g(\bar x-\bar y)| \leq {N-1\over |\bar x-\bar y| } + \omega_o \tau (1+ \tau)|\bar x -\bar y|^{-1+ \tau }\leq (N-1+ \omega_o \tau (1+ \tau)) |\bar x-\bar y|^{-1}$, $|\bar H| \leq {3\over 2} |D^2 g(\bar x-\bar y)|$   and then 
 (\ref{XY}) is nothing else but 
 \begin{equation} \label{XYlip}
  |X| + |Y| \leq  6M (|D^2 g(\bar x-\bar y)| + 1) \leq 6M(N-1+ \omega_o \tau (1+ \tau)) |\bar x-\bar y|^{-1}.
  \end{equation} 
 Furthermore  $q=M \omega^\prime (|\bar x-\bar y|) {\bar x-\bar y\over |\bar x-\bar y|} $
   $q^x =q+ 2M (\bar x-x_o), \            
                  q^y =q  -2M (\bar y-x_o)$.

Using (\ref{holder}) in $B_{1+r\over 2}$, for 
all $\gamma <1$, 
$$ M|\bar x-x_o|^2 + M |\bar y-x_o|^2 + \sup (u-v) \leq u(\bar x)-v(\bar y) \leq \sup (u-v) + c_{\gamma, r} |\bar x-\bar y|^\gamma$$ and then 

\begin{equation}\label{x-xo}
|\bar y-x_o|+  |\bar x-x_o|\leq \left({c_{\gamma,  r} |\bar x-\bar y|^\gamma\over M}\right)^{1\over 2}.
 \end{equation}  Then  taking 
 $\delta$ small enough,  more precisely if $(c_{ \gamma,  r} \delta^\gamma)^{1\over 2} < {1\over 4}$ by (\ref{omegaprime}), 
                                   
\begin{equation}\label{qx}
 {M\over 2} \leq |q| \leq M, \ 
                                   {M\over 4} \leq |q^x|, |q^y| \leq {5M\over 4}
\end{equation}
Then we derive from (\ref{ppvp})  that 
$ \Theta (X+Y-2(2M+1) {\rm I} ) \Theta$ has at least  one eigenvalue  less than  
                 \begin{equation}\label{c13}
                 - {\omega_o \tau (1+ \tau)\over 4}  M^{\alpha +1} |\bar x-\bar y|^{\tau-1}
                 \end{equation} 
Since $M |\Theta |^2 \leq  M^{1+\alpha} $,    (\ref{autresvp}) holds with 
$\tau_1 = 0 < 1-\tau$,  and $c_1 = 6$, while ({\ref{eqF}) is satisfied with 
$\tau_2 = -\gamma_F +1 < 1-\tau$, and $c_2 = c_{\gamma_F} (6+ 2\omega_o \tau (1+ \tau))$.

To check (\ref{eqqx}), by (\ref{ZT}),  (\ref{qx}),   (\ref{x-xo}),  and  (\ref{XYlip})
            $$ |\ |q_i^x|^{\alpha}-|q_i |^{\alpha} | \ |X_{ii}| \leq   (6+ 2\omega_o \tau (1+ \tau) ) M^{1+ {\inf (\alpha, 1)\over 2}}c_{  \gamma, r}^{\sup (1, \alpha) \over 2}    |\bar x-\bar y| ^{\inf (1, \alpha)  \gamma \over 2}  |\bar x-\bar y|^{-1}.$$ 
Hence,     by using $\inf (1,\alpha)  \gamma> 2 \tau$,   (\ref{eqqx}) holds with 
$\tau_3 = 1-{\inf(1, \alpha)\over 2} \gamma$ and $c_3 =2c_F (6+ 2\omega_o \tau (1+ \tau) )(c_{ \gamma,  r})^{\alpha\over 2} $ if $\alpha \leq 1$ and 
$c_3 = 2c_F(6+ 2\omega_o \tau (1+ \tau) c_{\gamma,  r}^{1\over 2} \alpha 3^{\alpha-1})$ if $\alpha \geq 1$. 

Finally $\tau_4 = 0 $ and $c_4 = c_{h, \Omega} (2^{1+\alpha} +1)$    are convenient for   (\ref{eqh}).         
 
 \medskip      
\noindent{\bf  Proofs of the claims   when $\omega(r)\simeq r$ and $\alpha \geq 2$}

In order to use the result in  Proposition \ref{prop4} we need  (\ref{eqNepsilon})  to be  satisfied.  For that aim  
we take $\tau$ and $\epsilon >0$ such that

 \begin{equation}\label{epsilonlip}
              0< \tau  <  \inf({1\over \alpha}, \gamma_F), \   1>\gamma > \tau\alpha, \ {\rm and }  \ 
              {\tau\over 2} < \epsilon<\inf ({{\gamma\over 2}-\tau\over  \alpha-2}, {\gamma_F-\tau\over \alpha-2}) . 
\end{equation} 
        Let us define $\omega$, $s_o$, as in the case $\alpha\leq 2$.             
            We   suppose $\delta < \delta_N$ where  
            \begin{eqnarray}\label{deltaN}
           \delta_N:&=&   \inf \left(   \exp {\log (\omega_o (1+\tau) \tau)-\log (2N(1+\omega_o \tau (1+ \tau)))\over 2\epsilon-\tau}, \right.\nonumber\\
          &&  \left. \exp {-\log (2 \omega_o (1+\tau))\over \tau}  \right)
            \end{eqnarray}
            In particular  since
there exists $i$ such that 
            $|\bar x_i-\bar y_i|^2 \geq {1\over N} |\bar x-\bar y|^2
\geq|\bar x-\bar y|^{2+ 2\epsilon} $, by (\ref{deltaN}),  $I(\bar
x-\bar y, \epsilon)\neq \emptyset$. 
             Furthermore, recall that  by (\ref{deltaN}), $1\geq
\omega^\prime (|\bar x-\bar y|) \geq {1\over 2}$ and  
             \begin{eqnarray*}
              {1\over 2} \omega^{\prime \prime } (|\bar x-\bar y|) &+&
                {N\over 2}  \omega_o \tau (1+ \tau) |\bar x-\bar
y|^{\tau-1+ 2\epsilon} +  {3\over 2} N |\bar x-\bar y|^{2\epsilon-1}
\omega^\prime (|\bar x-\bar y|) \\
                &\leq & {1\over 2} \omega^{\prime \prime } (|\bar
x-\bar y|)+ {N\over 2} ( \omega_o \tau (1+ \tau)  +3) |\bar x-\bar
y|^{2\epsilon-1} \\
                &\leq& 
              -{1\over 4} \omega_o (1+\tau) \tau |\bar x-\bar y|^{-1+
\tau }  =   { \omega^{\prime \prime } (|\bar x-\bar y|)\over 4}, 
            \end{eqnarray*}
and then (\ref{eqNepsilon}) holds. We still assume that (\ref{Mrlip}) holds.

As in the case $\alpha \leq 2$, using (\ref{holder}), one has, for $\delta$ small enough, 
(\ref{qx}) still holds. 
             
The hypothesis (\ref{deltaN}) ensures that $\Theta (X+Y -2(2M+1){\rm I}  ) \Theta$ has at least one eigenvalue less than 
             $$-{\omega_o \tau (1+\tau)\over 4}  M^{1+\alpha} |\bar x-\bar y|^{-1+ \tau+(\alpha-2) \epsilon}$$
and then using the fact that 
$\Theta(X+Y) \Theta \leq 6 M|\Theta |^2 \leq 6 M^{1+\alpha} $, 
by (\ref{epsilonlip}) and for $\delta$ small enough, (\ref{vpneg}) holds with $\hat \tau = (2-\alpha )\epsilon
+1-\tau$ and $c ={\omega_o \tau (1+\tau)\over 8}$   . Furthermore (\ref{autresvp} )  holds with $\tau_1 = 0$, and  $c_1 =6$ 
              
 As in the previous case, (\ref{XYlip}) holds, and  then (\ref{eqF}) holds with $\tau_2 = 1-\gamma_F< 1-\tau+ (2-\alpha) \epsilon$ and $c_2 =c_{\gamma_F} (6+ 2 \omega_o \tau (1+ \tau))$ .            
             
        Now using  (\ref{ZT}), (\ref{XYlip}), (\ref{qx}), (\ref{x-xo}),  one has 
        $$||q^x|^\alpha -|q|^\alpha | |X| \leq \alpha (M|\bar x-x_o|) ({5M\over 4})^{\alpha-1}M |\bar x-\bar y|^{-1} \leq  c_3  |\bar x-\bar y|^{{\gamma\over 2}-1}    M^{1+\alpha}$$
         and then (\ref{eqqx}) holds with 
         $\tau_3 = 1-{\gamma\over 2} < 1-\tau+ (2-\alpha) \epsilon$ and $c_3 = 2 (c_{\gamma,  r})(2)^{\alpha-1}$. 
           
Note finally that 
$$ |h(\bar x, q^x )| + | h(\bar y,  q^y)  |\leq 2c_h  \left({5M \over 4}\right)^{1+\alpha}$$
and then  (\ref{eqh}) holds with $\tau_4 = 0$ and $c_4 = 2^{2+\alpha}  c_h$  . 
\bigskip

 \section{Existence of solutions.}
 
  As it is classical, see e. g. \cite{usr}, the existence's Theorem \ref{corexi} will be proved once the following Propositions  are known:

   \begin{prop}\label{comp}  Suppose that $\Omega$ is a bounded domain  in $\R^N$ and that $F$ satisfies  (\ref{degenel}),  (H3) , (\ref{gammaF}), (\ref{diffp}). Suppose that $h$ is continuous and  it satisfies (\ref{H}).  
 Let $u$ be  a USC sub-solution  of 
$$ F(x, \nabla u, D^2 u) + h(x, \nabla u)-\beta (u)\geq f\ \mbox{in}\ \Omega$$
       and $v$ be  a LSC super-solution  of 
$$F(x, \nabla v, D^2 v) + h(x, \nabla v)-\beta (v) \leq g\ \mbox{in}\ \Omega$$
where $\beta$, $f$ and $g$ are   continuous.

Suppose that either  $\beta$ is increasing and   $f\geq  g$, or  $\beta$ is nondecreasing and  $f> g$. 
If  $u\leq v$ on $\partial \Omega$, 
then $u\leq v$ in $\Omega$.
\end{prop}

     \begin{prop}\label{exi}
      Suppose that the assumptions in Proposition \ref{comp} hold, and  that $f$ is continuous and bounded and $\beta$ is increasing.  
      If  $\underline{u}$ is   a USC sub-solution, and $\overline{u}$ is  a LSC super-solution of the equation
        
        $$F(x, \nabla u, D^2 u) + h(x, \nabla u) -\beta (u)= f, \   {\rm in } \ \Omega,$$
    such  that $\underline{ u}= \overline{ u}= \varphi$ on $\partial \Omega$. Then there exists $u$ a viscosity solution of the equation with 
         $\underline{ u}\leq u\leq \overline{ u}$ in $\Omega$, and $u = \varphi$ on $\partial \Omega$.
         \end{prop}
             The proofs of these two Propositions can be done by  using the classical tools, see  \cite{usr}. 
       \begin{rema}
        One can get the same existence's result when $\beta =0$, by using a standard approximation procedure and the stability of viscosity solutions. 
\end{rema}
       
       Nevertheless  the  proof of Theorem \ref{corexi} requires the existence of a super-solution which is zero on the boundary when $\beta = 0$ : 
         
         \begin{prop}\label{super}
 Suppose that $\Omega$ is a bounded ${\cal C}^2$ domain, and that $F$ and $h$ satisfy the hypothesis in Proposition \ref{exi}. Then for any $f$ continuous and bounded,   there exist a super-solution and a sub-solution of  of 
$$ F(x, \nabla u, D^2 u) + h(x, \nabla u)=  f\ \mbox{in}\ \Omega$$ which  are  zero on the boundary. 
\end{prop}

\noindent{\em Proof of Proposition \ref{super} }: 

Let us recall that the distance to the boundary $d$ satisfies 
everywhere: $d$ is semi concave or equivalently  there exists $C_1$ such that 
                           $$D^2 d \leq C_1 {\rm I}. $$ 
                           In the following lines we will make the computations as if $d$ is ${\cal C}^2$, it is not difficult to see that the required inequalities  hold  also in the viscosity sense. 
                           We now choose $k$ large such that 
                           $$ (k+1) (\sum_1^N |\partial_i d|^{2+\alpha})^{1+\alpha\over 2+\alpha}  \geq 2C_1 N^{1+\alpha\over 2+\alpha}(1+ {\rm diam} \ \Omega). $$
                           This  can be done since 
                            $\sum_1^N(\partial_i d)^2 = 1 \leq (\sum_1^N |\partial_i d|^{2+\alpha})^{2\over 2+\alpha} N^{1+\alpha \over 2+\alpha}$. 
                            We will  choose later $M$ large and define 
                            $$\psi(x) = M (1-{1\over (1+ d)^k}). $$
                             Clearly 
$$\grad \psi = M {k \grad d  \over (1+ d)^{k+1}}, \  D^2 \psi ={ M k\over (1+ d)^{k+2}} ((1+ d) D^2d - (k+1) \grad d\otimes \grad d)$$
and then  choosing  $k$ such that $\lambda(k+1)N^{-\frac{1}{2+\alpha}}\geq 3N\Lambda C_1+ 2C_h(1+ {\rm diam} \ \Omega) $.
\begin{eqnarray*}
F(x, \grad\psi,D^2\psi)+ h(x, \nabla \psi) &\leq& {(Mk)^{\alpha+1}\over (1+d)^{k+2+ (k+1) \alpha}} \left((1+ d){\cal M}_{\alpha}^+(\grad d,D^2 d)\right.
\\
&&-\left.(k+1) {\cal M}_{\alpha}^-(\nabla d, D^2 d)\right)
+ C_h |\nabla \psi|^{1+\alpha} \\
 &\leq & {(Mk)^{\alpha+1}\over (1+d)^{k+2+ (k+1) \alpha}}[(1+ d)\Lambda C_1\sum|\partial_i d|^\alpha\\
 &&-(k+1)\lambda\sum|\partial_i d|^{\alpha+2}]
 + C_h {(Mk)^{\alpha+1} \over (1+d)^{(k+1)(1+\alpha)}}\\ 
 &\leq&  {(Mk)^{\alpha+1}\over (1+d)^{k+2+ (k+1) \alpha}}(2N\Lambda C_1-\lambda(k+1)N^{-\frac{1}{2+\alpha}})\\
&& +C_h {(Mk)^{\alpha+1}\over (1+d)^{(k+1)(1+\alpha)}}\\
 &\leq &  -{(k+1)\lambda N^{-\frac{1}{2+\alpha}}(Mk)^{\alpha+1}\over 4(1+d)^{k+2+ (k+1) \alpha}}. 
\end{eqnarray*}

It is clear that one can choose $M$ large enough as soon as $k$ is fixed as above in order that 
$$ F(x, \grad\psi,D^2\psi)+ h(x, \nabla \psi)\leq -\|f\|_\infty.$$
A similar computation leads to:
$$F(x, \grad(-\psi),D^2(-\psi)) +h(x, \nabla -\psi)\geq \|f\|_\infty.$$

\section{The strong Maximum Principle}

\begin{theo}
Suppose that $u$ is a supersolution  of the 
equation $F(x, \nabla u, D^2 u) \leq 0$ in a domain $\Omega $  and that $u\geq 0$. 
Then either $u>0$ in $\Omega$ or $u\equiv 0$.
\end{theo}
\begin{proof} 
One can suppose that $u>0$, on $B(x_1, |x_1-x_o|) $, $u(x_o)=0$, $R = |x_1-x_o|$ and we can assume  that  the annulus  $ {R\over 2} \leq |x-x_1| \leq {3R\over 2}$ is included in $\Omega$.  Let $w$  be defined as $$w(x) = m(e^{-c|x-x_1|}-e^{-cR})$$
for some $c$ and $m$ to be chosen. Without loss of generality we will suppose that $x_1=0$  and denote
$r:=|x-x_1|= |x|$.  We choose $m$ so that on $r= {R\over 2}$, $w \leq u$.   
In the sequel  for simplicity we replace $m$ by $1$. 
                           
One has
$$\nabla w = {-c x\over r} e^{-cr},\  D^2 w =e^{-cr}  ({c^2 \over r^2}+ {c\over r^3} ) (x\otimes  x  ) -{c\over r} {\rm I}$$
and then, using the usual notation $\Theta(\nabla w) $, 
$
H:= \Theta(\nabla w) D^2 w \Theta(\nabla w) $, 
i.e.                            
$$
He^{c(\alpha+1) r}
=\left( {c\over r}\right)^\alpha \left( ({c^2 \over r^2}+ {c\over r^3} )\vec i\otimes \vec i-{c\over r} \vec j\otimes \vec j\right)
                     $$
where  $\vec i = \sum |x_i|^{\alpha \over 2} x_i e_i$ and $\vec j =  \sum |x_i|^{\alpha \over 2} e_i$.
                              
We need to evaluate the eigenvalues of $H$ and in particular prove that 
                               $$ {\cal M} ^-(H) > 0.$$
For that aim let us note that $(\vec i, \vec j)^\perp$ is in the kernel of $H$. 
We  introduce $ a= {c^2\over r^2}+ {c\over r^3}$ and $b = -{c\over r}$. Then the 
non zero eigenvalues 
of ${H c^{- \alpha} e^{cr (1+ \alpha)}}$ are given by 
 $$\mu^{\pm} = {a|\vec i|^2 + b |\vec j|^2 \over 2} \pm \sqrt{\left({a|\vec i|^2 + b |\vec j|^2 \over 2}\right)^2 - ab (|\vec i|^2 |\vec j|^2 -(\vec i \cdot \vec j)^2)}.$$
 Note that there exist constants $c_i(N, \alpha)$ for $i=1,\cdots 4$,  such that 
$$ c_1(N, \alpha) \left({R\over 2}\right) ^{\alpha+2}\leq c_1(N, \alpha) r^{\alpha+2} \leq |\vec i|^2\leq c_2 (N, \alpha) r^{\alpha+2} \leq  c_2 (N, \alpha) \left({3R\over 2}\right)^{\alpha+2} $$

and  
$$ c_3(N, \alpha) \left({R\over 2}\right) ^{\alpha}\leq c_3(N, \alpha) r^{\alpha} \leq |\vec j|^2\leq c_4 (N, \alpha) r^{\alpha } \leq  c_4 (N, \alpha) \left({3R\over 2}\right)^{\alpha}. 
$$
Note that one can choose $c$ large enough in order that  for some constant $c_5(N, \alpha)$

\begin{eqnarray*}
                               a|\vec i|^2 +  b |\vec j|^2 &\geq&  c_1(N, \alpha) \left({R\over 2}\right) ^{\alpha+2}{c^2\over r^2} -  c_4 (N, \alpha) \left({3R\over 2}\right)^{\alpha} {c\over r}\\
                               & \geq &c_5 (N, \alpha) c^2.
\end{eqnarray*}
On the other hand one can assume $c$ large enough in order that  
\begin{eqnarray*}
 4|ab| (|\vec i|^2 |\vec j|^2 -(\vec i \cdot \vec j)^2) &\leq& 4{c^3\over r^2}c_2(N, \alpha) c_4(N, \alpha) \left({3R\over 2}\right)^{2\alpha+2} \\
 &\leq & c_6(N, \alpha) c^3 \\
 &<  &[\left({\lambda+ \Lambda \over \Lambda-\lambda}\right)^2-1]( c_5 (N, \alpha) c^2))^2\\
 &\leq  &[\left({\lambda+ \Lambda \over \Lambda-\lambda}\right)^2-1]\left(  a|\vec i|^2 +  b |\vec j|^2\right)^2.
 \end{eqnarray*}
 In particular this implies 
 \begin{eqnarray*}
\lambda \mu^++\Lambda \mu^- &= &( {{a|\vec i|^2 +  b |\vec j|^2}\over 2} ) \left((\lambda + \Lambda) +( \lambda - \Lambda)   \sqrt{1+ 4 {|ab| (|\vec i|^2 |\vec j|^2 -(\vec i \cdot \vec j)^2) \over (a|\vec i|^2 +  b |\vec j|^2)^2}}\right)>0\\
                               \end{eqnarray*}
i.e.
${\cal M}^-(H) >0.$
Using the comparison principle in the annulus $\{ \frac{R}{2} \leq |x-x_1| \leq {3R\over 2}\}$ one 
obtains that $u\geq w$. 

Observe   that  $w$ touches $u$ 
by below on $x_o$, and then,  since $w$ is ${\cal C}^2$  around $x_o$,  by the definition of viscosity solution
$$F(x_o, \nabla w(x_o), D^2 w(x_o)) \leq 0.$$
This contradicts the above computation. 
\end{proof}

 \begin{rema}
As it is well known, the above proof can be used to see that on a point of the boundary where the interior sphere condition is satisfied, the Hopf principle holds. 
  \end{rema}

 \section{Appendix: Proof of Lemma \ref{lem2}}
The proof of Lemma \ref{lem2} is based on the following Lemma by Ishii
\begin{lemme}[Ishii]\label{lem1}
Let $A$ be  a symmetric matrix on $\R^{2N}$. Suppose that 
$U \in USC (\R^N)$ and $V\in USC (\R^N)$ satisfy 
$U(0)= V(0)$ and, for all $(x,y)\in( \R^N)^2$,
$$U(x)+ V(y) \leq {1\over 2} (^tx,^ty)A \left(\begin{array}{c}
         x\\
         y\end{array}\right).$$
Then, for all $\iota>0$, there exist $X^U_\iota \in S$, $X^V_\iota \in S$ such that 
          $$(0, X^U_ \iota) \in \bar J^{2,+} U(0), \ (0, X^V_\iota)\in \bar J^{2,+} V(0)$$
and 
$$-({1\over \iota} + |A|) \left(\begin{array}{cc}
           I &0\\
           0&I\end{array}\right) \leq \left(\begin{array}{cc}
           X^U_\iota&0\\
           0& X^V_\iota
           \end{array}\right)\leq (A+\iota A^2).
$$
           \end{lemme}

We can now start the proof of Lemma \ref{lem2}. 
The second order Taylor's expansion  for $\Phi$,  gives that for all 
$\epsilon  >0$  there exists   $r>0$ such that, for $|x-\bar x| ^2 + |\bar y -y|^2 \leq r^2$,
\begin{eqnarray*}
 u(x)-u(\bar x)&-&\langle D_1\Phi(\bar x, \bar y)+2M(\bar x-x_o), x-\bar x\rangle+\\
+v(\bar y)-v(y) & -& \langle D_2 \Phi(\bar x, \bar y)+2M(\bar y-x_o)  , y-\bar y\rangle  \\
 &\leq &{1\over 2} \left(^t (x-\bar x), ^t (y-\bar y) \right)(D^2 \Phi(\bar x, \bar y)+\epsilon  {\rm I} )  \left( \begin{array}{c} x-\bar x\\
                y-\bar y\end{array}\right)\\
                && +M   (|x-\bar x|^2 + |y-\bar  y|^2). 
                \end{eqnarray*}
We now introduce the functions $U$ and $V$ defined, in the  closed ball  $|x-\bar x| ^2 + |\bar y -y|^2 \leq r^2$,  by 
$$U(x) = u(x+ \bar x)-\langle D_1\Phi(\bar x, \bar y)+2M(\bar x-x_o), x\rangle -u(\bar x) -M |x|^2$$
and
$$ V(y) = -v(y+ \bar y) - \langle D_2 \Phi(\bar x, \bar y)+2M(\bar y-x_o)  , y\rangle + v(\bar y)-M |y|^2 $$
which we extend  by some
convenient negative constants   in the complementary of that ball (see  \cite{I1} for details). Observe first that 
                
$$ (0, X^U) \in \overline{J}^{2,+} U(0), \ (0, X^V ) \in \overline{J}^{2,-} V(0)$$ is equivalent to 
                
$$(D_1 \Phi(\bar x, \bar y) + 2M (\bar x-x_o), X^U + 2M{\rm I}) \in  \overline{J}^{2,+} u(\bar x)$$
and 

$$ (-D_2 \Phi (\bar  x, \bar y)-2M (\bar y-x_o), -X^V - 2M{\rm I} )\in \overline{J}^{2,-} v(\bar y).$$
We can apply Lemma \ref{lem1}, which gives that, for any $\iota>0$, there exists 
$(X_\iota, Y_\iota) $ such that 
                  
$$(D_1 \Phi(\bar x, \bar y)+2M(\bar x-x_o), X_\iota )\in \bar J^{2,+} u(\bar x)$$
and
$$(-D_2 \Phi(\bar x, \bar y) -2M (\bar y-x_o), -Y_\iota ) \in \bar J^{2,-}v(\bar y)$$
Choosing $\epsilon $ such that  $2 \epsilon  \iota |D^2 \Phi(\bar x, \bar y)| +  \epsilon+ \iota (\epsilon )^2 < 1$, one gets     
\begin{eqnarray*}
-({1\over \iota} + |D^2 \Phi|+1 ) \left(\begin{array}{cc}
           I &0\\
           0&I\end{array}\right) &\leq& \left(\begin{array}{cc}
           X_\iota-2M{\rm I}&0\\
           0& Y_\iota-2M {\rm I}
           \end{array}\right)\\
           &\leq& (D^2 \Phi+\iota (D^2 \Phi)^2)+  \left(\begin{array}{cc}
           I &0\\
           0&I\end{array}\right)  .
\end{eqnarray*}


\begin{thebibliography}{99}
 
  \bibitem{BCI} G. Barles,  E. Chasseigne, C. Imbert, {\it H\"older continuity of solutions of second-order non-linear elliptic integro-differential equations},  J. Eur. Math. Soc. 
 vol {13} (2011),  p 1-26.

\bibitem{BEQ} J. Busca,M.J. Esteban, A. Quaas,   {\it  Nonlinear eigenvalues and bifurcation problems for PucciÕs operator}, Annales de lÕInstitut H. Poincar\'e, Analyse non-lin\'eaire 22 (2005) 187Ð206.
\bibitem{BK} M. Belloni B. Kawohl  {\em The  Pseudo p-Laplace  eigenvalue problem and viscosity solutions},  
ESAIM COCV, vol. 10, (2004), p 28-52. 

\bibitem{BNV} H. Berestycki, L.  Nirenberg, S.R.S. Varadhan, {\it  The principal 
eigenvalue and maximum principle for second-order elliptic operators in general domains.}
  Comm. Pure Appl. Math.  47  (1994),  no. 1, 47-92.
\bibitem{BD1} I. Birindelli, F. Demengel {\em First eigenvalue and Maximum principle for fully nonlinear singular
operators} ,   Advances in Differential equations, Vol
11, Number 1, (2006) p. 91-119. 
\bibitem{BD2} I. Birindelli, F. Demengel, {\em  ${\cal C}^{1,\beta}$ regularity for Dirichlet problems associated to fully nonlinear degenerate elliptic 
equations},  ESAIM  COCV, vol  20, Issue 40, (2014), p. 1009-1024.  et  https://hal.archives-ouvertes.fr.hal-01076713.


\bibitem{BBJ}
P. Bousquet,   L. Brasco, V. Julin, {\em Lipschitz regularity for local minimizers of some widely degenerate problems},  Calc. of Variations and Geometric Measure theory. http://cvgmt.sns.it/paper/2515/
 
 
\bibitem{BC} L.  Brasco, G. Carlier, {\em
   On certain anisotropic elliptic equations arising in congestion optimal transport : Local gradient bounds }, 
  Advances in Calculus of Variations, Vol 7,  (2014), Issue 3,p 379-407. 







\bibitem{usr} { M.G. Crandall,  H. Ishii, P.L. Lions} {\em User's guide to
viscosity solutions of second order partial differential equations}, Bull.
Amer. Math. Soc. (N.S.) 27 (1992), no. 1, p 1-67.

\bibitem{D} F. Demengel, {\em Lipschitz interior  regularity for the  viscosity  and weak solutions of the Pseudo $p$-Laplacian Equation.} Advances in Differential Equations,  Vol. 21, Numbers 3-4, 2016.

\bibitem{DB} E. Di Benedetto,  {\em ${\cal C}^{1+\beta}$  local regularity of weak solutions of degenerate elliptic equations,} Nonlinear  Analysis, Theory,  Methods  and  Applications, Vol. 7. No. 8. pp. 827-850, 1983.

\bibitem{D} F Demengel, {\em Lipschitz interior  regularity for the  viscosity  and weak solutions of the Pseudo $p$-Laplacian Equation} to appear in Advances in Differential equations. 
\bibitem{FFM} I. Fonseca, N. Fusco,  P. Marcellini, 
{\em An existence result for a non convex variational problem via regularity}, ESAIM: Control, Optimisation and Calculus of Variations, Vol. 7, ( 2002),  p 69-95. 


\bibitem{I} C. Imbert, {\em  Alexandroff-Bakelman-Pucci estimate and Harnack inequality for degenerate fully non-linear elliptic equations}  J. Differential Equations 250 (2011), no. 3, 1553-1574. 


\bibitem{IS} C. Imbert, L. Silvestre,  {\em ${\cal C}^{1, \alpha}$   regularity of solutions of degenerate fully non-linear elliptic equations} , Adv. Math. vol 233, (2013), p 196-206.

\bibitem{I1} {H. Ishii},  {\it Viscosity solutions  of Nonlinear fully nonlinear equations} Sugaku Expositions , Vol 9, number 2, December 1996. 
\bibitem{IL}{H. Ishii, P.L. Lions}, {\it  Viscosity solutions of Fully-Nonlinear Second  Order Elliptic Partial Differential Equations}, {J. Differential Equations},  vol 83,   (1990),  p 26-78.


\bibitem{T} P. Tolksdorff {\em Regularity for a more general class of quasilinear elliptic
equations}, J. Differential Equations, 51 (1984), 126-150.


           \end{thebibliography}
  \end{document}